\documentclass[11pt]{amsart}
\usepackage{amsmath}
\usepackage{amsthm}
\usepackage{amsfonts} 

\usepackage{mathabx}
\usepackage{amssymb}
\usepackage[english]{babel} 
\usepackage[colorlinks = true,
            linkcolor = blue,
            urlcolor  = blue,
            citecolor = blue,
            anchorcolor = blue]{hyperref} 
            
\usepackage{cleveref}
\usepackage{calc} 
\usepackage{dsfont}
\usepackage{enumitem}
\usepackage{placeins}

\usepackage{graphicx}
\usepackage{subcaption}

\usepackage{float}

\usepackage{physics}
\usepackage{tikz}

\allowdisplaybreaks
\usepackage[a4paper, lmargin=0.1666\paperwidth, rmargin=0.1666\paperwidth, tmargin=0.1111\paperheight, bmargin=0.1111\paperheight]{geometry} 
\usepackage[all]{nowidow}
\usepackage{lipsum}

\theoremstyle{plain}
\newtheorem{theorem}{Theorem}[section]
\newtheorem{conjecture}[theorem]{Conjecture}

\newtheorem{lemma}[theorem]{Lemma}

\theoremstyle{definition}
\newtheorem{definition}[theorem]{Definition}

\theoremstyle{remark}

\newtheoremstyle{named}{}{}{\itshape}{}{\bfseries}{.}{.5em}{\thmnote{#3}#1}
\theoremstyle{named}

\setlist[enumerate]{itemsep=0pt, topsep=2pt}

\usepackage{algorithm}
\usepackage{algpseudocode}

\newcommand{\E}{\mathbb E}
\newcommand{\Pp}{\mathbb P}

\title{On Feige's conjecture}

\author{Zipei Nie} 
\address{Department of Mathematics\\ University of Illinois at Urbana--Champaign\\ Urbana, IL 61801, USA} \email{znie@illinois.edu} 

\author{Jiaye Wei} \address{Chair of Discrete Optimization\\ Institute of Mathematics\\ École Polytechnique Fédérale de Lausanne \\ 1015 Lausanne, Switzerland} \email{jiaye.wei@epfl.ch}
	
\date{\today}

\begin{document}

\begin{abstract} We present a short proof of Feige's conjecture: for \(n\) independent nonnegative random variables with expectation one, the probability that their sum is less than \(n+1\) is at least \(\left(\frac{n}{n+1}\right)^n\ge \frac{1}{e}\). The proof was obtained with the assistance of GPT-5.6 Sol and builds on the recent breakthrough of Vlassis and Thomas establishing Gaffke's conjecture on the finite-sample validity of a distribution-free \(p\)-value. We also discuss the implications of the subsequent work of Ming, Ramdas, Shen, Wang, and Waudby-Smith. \end{abstract}

\maketitle
\section{Introduction}
For independent nonnegative random variables \(X_1,\ldots,X_n\) with respective expectations \(\mu_1,\ldots,\mu_n\leq 1\), Feige \cite{feige2006sums} proved the concentration inequality
\begin{equation}\label{eq: Feige}
 \Pp\left(\sum_{i=1}^n X_i<\sum_{i=1}^n \mu_i+\delta\right)\geq \min\left( \frac{\delta}{1+\delta}, \frac{1}{13} \right)
\end{equation}
for every \(\delta>0\), without imposing any assumptions on the variances of the \(X_i\). 

Because it provides a dimension-free lower-tail bound under only independence and nonnegativity, Feige's inequality has found applications across a wide range of areas, including randomized graph algorithms \cite{feige2006sums}, extremal combinatorics \cite{alon2012nonnegative,frankl2022erdHos}, randomized search heuristics \cite{dang2015simplified,dang2019level,lissovoi2019time}, and the study of innovation diffusion in social networks \cite{arieli2020speed}.

Beyond its applications, Feige's inequality is also of intrinsic probabilistic interest. Feige \cite{feige2006sums} conjectured that the constant \(\frac{1}{13}\) in \eqref{eq: Feige} can be replaced by \(\frac{1}{e}\) for every \(\delta>0\).

\begin{conjecture}[Feige's conjecture for arbitrary \(\delta\)]\label{conj:main-2}
Let \(X_1,\ldots,X_n\) be independent nonnegative random variables with \(\E \left(X_i\right)=\mu_i\le 1\) for each $i=1,\ldots, n$. Then, for every \(\delta>0\),
\[
\Pp\left(\sum_{i=1}^n X_i<\sum_{i=1}^n \mu_i+\delta\right)\ge\min\left(\frac{\delta}{1+\delta},\frac{1}{e}\right).
\]
\end{conjecture}

By replacing \(X_i\) with \(X_i+1-\mu_i\), one may equivalently assume that all the random variables have expectation one. The case \(\delta=1\) is commonly known as Feige's conjecture.

\begin{conjecture}[Feige's conjecture] \label{conj:main}
Let \(X_1,\ldots,X_n\) be independent nonnegative random variables with expectation one. Then
\[
\Pp\left(\sum_{i=1}^n X_i<n+1\right)\ge\frac{1}{e}.
\]
\end{conjecture}

Prior to the present work, the constant in Feige's conjecture was successively improved from \(\frac{1}{13}\) to \(\frac{1}{8}\) by He, Zhang, and Zhang \cite{he2010bounding}, to \(\frac{7}{50}\) by Garnett \cite{garnett2020small}, and to \(0.1798\) by Guo, He, Ling, and Liu \cite{guo2020bounding}, which was the best previously known result. Related partial results were obtained in \cite{paulin2017some,alqasem2024conjecture,egozcue2025short}.

Our main result is the following bound for every \(0<\delta\le 1\).

\begin{theorem}\label{thm:main} Let \(X_1,\ldots,X_n\) be independent nonnegative random variables with expectation one. Then, for every \(0<\delta\leq 1\), 
\[
\Pp\left(\sum_{i=1}^n X_i<n+\delta\right) \geq \delta\left(\frac{n}{n+\delta}\right)^n. \]
\end{theorem} 

When \(\delta=1\), Theorem~\ref{thm:main} gives the sharp bound for every \(n\), and hence proves Feige's conjecture (Conjecture~\ref{conj:main}). Sharpness follows by taking each \(X_i\) to be equal to \(n+1\) with probability \(\frac{1}{n+1}\) and to \(0\) otherwise. The arbitrary-\(\delta\) version of Feige's conjecture (Conjecture~\ref{conj:main-2}) remains open.

The paper is organized as follows. In Section~\ref{sec:2}, we prove Theorem~\ref{thm:main} using the finite-sample validity result \cite{vlassis2026exact} of Vlassis and Thomas. In Section~\ref{sec:3}, we discuss the subsequent work \cite{ming2026gaffke} of Ming, Ramdas, Shen, Wang, and Waudby-Smith and its implications.

\subsection*{Acknowledgement}
The authors acknowledge the use of GPT-5.6 Sol in the discovery and exploration of the proof. All arguments were independently verified by the authors. The manuscript was written entirely by the authors, who take full responsibility for its content.

\section{Proof of Theorem~\ref{thm:main}}
\label{sec:2}

In this section, we prove Theorem~\ref{thm:main}. We begin by recalling the definition of a merger from \cite{ming2026gaffke} and recording a simple reduction.

\begin{definition}[Merger]
A Borel measurable function \(F:[0,\infty)^n\to[0,1]\) is called a merger if, for every collection of independent nonnegative random variables \(X_1,\ldots,X_n\) satisfying \(\E(X_i)\leq 1\) for all \(i\), and every \(\alpha\in[0,1]\),
\[
\Pp\left(F(X_1,\ldots,X_n)\leq\alpha\right)\leq\alpha.
\]
\end{definition}

The following lemma reduces the proof to a pointwise bound on a merger.

\begin{lemma}\label{lemma:reduction} Let $\delta>0$, let \(F\) be a merger, and let \(\alpha\in[0,1]\). Suppose that \[ F(x_1,\ldots,x_n)\leq\alpha \] whenever \(x_1,\ldots,x_n\geq 0\) and \[ \sum_{i=1}^n x_i\geq n+\delta. \] Then, for any independent nonnegative random variables \(X_1,\ldots,X_n\) with expectation one, \[ \Pp\left(\sum_{i=1}^n X_i<n+\delta\right)\geq 1-\alpha. \] \end{lemma}
\begin{proof} By assumption, \[ \left\{\sum_{i=1}^n X_i\geq n+\delta\right\} \subseteq \left\{F(X_1,\ldots,X_n)\leq\alpha\right\}. \] Since \(F\) is a merger and the \(X_i\) are independent nonnegative random variables with expectation one, 
we have \begin{align*}
    &\Pp\left(\sum_{i=1}^n X_i<n+\delta\right)\\
    =&1-\Pp\left(\sum_{i=1}^n X_i\ge n+\delta\right)\\
    \ge &1-\Pp\left(F(X_1,\ldots,X_n)\leq\alpha\right)\\
    \ge &1-\alpha.
\end{align*}
\end{proof}

We apply this reduction to the merger $K_n$ introduced in the recent work \cite{vlassis2026exact} of Vlassis and Thomas.

\begin{definition}[Function \(K_n\)] Let \((D_0,D_1,\ldots,D_n)\sim\operatorname{Dir}(1,1,\ldots,1)\) be a vector of Dirichlet weights. For \(x=(x_1,\ldots,x_n)\in[0,\infty)^n\), define \[ K_n(x_1,\ldots,x_n) = \Pp\left(\sum_{i=1}^n x_iD_i\leq 1\right). \] 
\end{definition}

\begin{theorem}\cite{vlassis2026exact}\label{thm:merger-1}
    The function \(K_n\) is a merger for every \(n\geq 1\).
\end{theorem}

It therefore remains to bound \(K_n\) on the region where the sum of its arguments is at least \(n+\delta\). The required estimate follows from a geometric inequality for half-space sections of a simplex.

\begin{lemma}\label{lem:pointwise-1}
For each \(0<\delta\leq 1\) and each
\((x_1,\ldots,x_n)\in[0,\infty)^n\) satisfying
\[
\sum_{i=1}^n x_i\geq n+\delta,
\]
we have
\[
K_n(x_1,\ldots,x_n)
\leq
1-\delta\left(\frac{n}{n+\delta}\right)^n.
\]
\end{lemma}

\begin{proof}
Let \(P\) denote the hyperplane
\[
P=
\left\{
(d_0,\ldots,d_n)\in\mathbb{R}^{n+1}:
\sum_{i=0}^n d_i=1
\right\},
\]
and let \(\Delta\) denote the \(n\)-simplex
\[
\Delta=
\left\{
(d_0,\ldots,d_n)\in[0,\infty)^{n+1}:
\sum_{i=0}^n d_i=1
\right\}.
\]
The centroid of \(\Delta\) is
\(
\left(\frac{1}{n+1},\ldots,\frac{1}{n+1}\right).
\)
By the density formula for the Dirichlet distribution,
\((D_0,\ldots,D_n)\) is uniformly distributed on \(\Delta\).

Let
\[
u=
\left(
-\frac{\sum_{i=1}^n x_i}{n+1},
x_1-\frac{\sum_{i=1}^n x_i}{n+1},
\ldots,
x_n-\frac{\sum_{i=1}^n x_i}{n+1}
\right)
\]
and
\[
v=
\left(
D_0-\frac{1}{n+1},
D_1-\frac{1}{n+1},
\ldots,
D_n-\frac{1}{n+1}
\right).
\]
Both \(u\) and \(v\) are parallel to \(P\), and \(u\) is nonzero. Moreover,
\begin{align*}
\langle u,v\rangle
&=
-\frac{\sum_{i=1}^n x_i}{n+1}
\left(D_0-\frac{1}{n+1}\right)
+
\sum_{i=1}^n
\left(x_i-\frac{\sum_{i=1}^n x_i}{n+1}\right)
\left(D_i-\frac{1}{n+1}\right) \\
&=
-\frac{\sum_{i=1}^n x_i}{n+1}D_0
+\sum_{i=1}^n x_iD_i
-\frac{\sum_{i=1}^n x_i}{n+1}\sum_{i=1}^n D_i \\
&=
\sum_{i=1}^n x_iD_i-\frac{\sum_{i=1}^n x_i}{n+1}.
\end{align*}
As \((D_0,\ldots,D_n)\) ranges over \(\Delta\), the minimum of
\(\langle u,v\rangle\) is
\(
-\frac{\sum_{i=1}^n x_i}{n+1},
\)
attained at \((1,0,\ldots,0)\). Since \(\sum_{i=1}^n x_i\ge n+\delta\), we have
\[  \left\{\langle u,v\rangle >\frac{1-\delta}{n+\delta}\frac{\sum_{i=1}^n x_i}{n+1} \right\} \subseteq  \left\{\sum_{i=1}^nx_iD_i>1\right\}.\]
Therefore, \cite[Theorem~4]{letwin2024generalization}, applied to the
centered simplex, gives
\begin{align*}
K_n(x_1,\ldots,x_n)
&=1-
\Pp\left(\sum_{i=1}^n x_iD_i>1\right) \\
&\le1-
\Pp\left(\langle u,v\rangle >\frac{1-\delta}{n+\delta}\frac{\sum_{i=1}^n x_i}{n+1}\right) \\
&\le 1-
\left(\frac{n}{n+1}\right)^n
\left(1+\frac{1-\delta}{n+\delta}\right)^{n-1}
\left(1-n\frac{1-\delta}{n+\delta}\right) \\
&=1-
\delta\left(\frac{n}{n+\delta}\right)^n.
\end{align*}
\end{proof}

We can now complete the proof of the main theorem.

\begin{proof}[Proof of Theorem~\ref{thm:main}] Set \[ \alpha = 1-\delta\left(\frac{n}{n+\delta}\right)^n. \] Since \(0<\delta\leq 1\), we have \(\alpha\in[0,1]\). By Lemma~\ref{lem:pointwise-1}, we have \[ K_n(x_1,\ldots,x_n)\leq\alpha \] whenever \(x_1,\ldots,x_n\geq 0\) and \[ \sum_{i=1}^n x_i\geq n+\delta. \] Thus, by Lemma~\ref{lemma:reduction} and Theorem~\ref{thm:merger-1}, we have \[ \Pp\left(\sum_{i=1}^n X_i<n+\delta\right) \geq 1-\alpha = \delta\left(\frac{n}{n+\delta}\right)^n. \] 
\end{proof}

\section{The merger \(K_{2}^{\mathrm{ad}}\)}
\label{sec:3}

Although the merger \(K_n\) allows us to prove Feige's conjecture (Conjecture~\ref{conj:main}), the bound obtained from \(K_n\) is not sharp for arbitrary \(\delta\) (Conjecture~\ref{conj:main-2}). In recent work, Ming, Ramdas, Shen, Wang, and Waudby-Smith \cite{ming2026gaffke} studied whether \(K_n\) can be improved as a merger. 

A merger \(G\) is said to dominate a merger \(F\) if \(G\leq F\) pointwise, and to strictly dominate \(F\) if the inequality is strict at some point. A merger is admissible if it is not strictly dominated by any other merger.

For \(n=2\), they constructed a merger \(K_2^{\mathrm{ad}}\), which is the unique admissible merger among those that dominate \(K_2\). This is particularly relevant here, since Lemma~\ref{lemma:reduction} shows that a smaller merger may yield a stronger lower-tail bound. We now recall their construction.

\begin{definition}[Function $K_2^{\mathrm{ad}}$]
For \(x,y\geq 0\), let \(a=\min(x,y)\) and \(b=\max(x,y)\). When \(0\leq a<1<b\), let \(\tau(a,b)\) be the larger root of
\[
t^2-b(1+a)t+ab=0.
\]
Define \(K_2^{\mathrm{ad}}\colon[0,\infty)^2\to[0,1]\) by
\[
K_2^{\mathrm{ad}}(a,b)
=
\begin{cases}
1,
& b\leq 1,\\
\frac{2}{\tau(a,b)}-\frac{1}{\tau(a,b)^2},
& 0\leq a<1<b,\\
\frac{1}{ab},
& 1\leq a.
\end{cases}
\]
\end{definition}

\begin{theorem}\cite[Theorem~3.6]{ming2026gaffke}\label{thm:merger-2}
The function \(K_2^{\mathrm{ad}}\) is a merger.
\end{theorem}

The improved merger satisfies the following pointwise estimate.

\begin{lemma}\label{lem:pointwise-2}
For each \(\delta>0\) and each \(x,y\geq 0\) satisfying
\[
x+y\geq 2+\delta,
\]
we have
\[
K_2^{\mathrm{ad}}(x,y)
\leq
1-
\min\left(
\frac{\delta}{1+\delta},
\left(\frac{1+\delta}{2+\delta}\right)^2
\right).
\]
\end{lemma}

\begin{proof}
Let \(a=\min(x,y)\) and \(b=\max(x,y)\). Since $$2b \ge a+b=x+y\ge 2+\delta>2,$$ we have $b>1$. 

If $a\ge 1$, then 
$$K_2^{\mathrm{ad}}(a,b)=\frac{1}{ab}\le \frac1{a+b-1}
    \le\frac{1}{1+\delta}
    \le 1-
\min\left(
\frac{\delta}{1+\delta},
\left(\frac{1+\delta}{2+\delta}\right)^2
\right).$$

Suppose that $0\le a<1<b$. By \cite[Lemma~3.4]{ming2026gaffke}, \[
K_2^{\mathrm{ad}}(x,y)
\leq
1-
\min\left(
\frac{\delta}{1+\delta},
\left(\frac{1+\delta}{2+\delta}\right)^2
\right)\] if and only if \[ b\ge \frac{1}{(1-s)(1+as)}\] where $$s= \min\left(
\sqrt{\frac{\delta}{1+\delta}},
\frac{1+\delta}{2+\delta}
\right)<1.$$

Since a concave function can only attain its minimum at a boundary point, we have 
\begin{align*}
    &b(1-s)(1+as)\\
    \ge &(2+\delta -a)(1-s)(1+as)\\
    \ge &\min\left((2+\delta)(1-s),(1+\delta)(1-s^2)\right)\\
    \ge&1.
\end{align*}
This proves the result.
\end{proof}

Combining this estimate with the merger property gives the sharp two-variable form of Feige's conjecture for every \(\delta>0\).

\begin{theorem}\label{thm:n2-arbitrary-delta}
Let \(X_1\) and \(X_2\) be independent nonnegative random variables with expectation one. Then, for every \(\delta>0\),
\[
\Pp\left(X_1+X_2<2+\delta\right)
\geq
\min\left(
\frac{\delta}{1+\delta},
\left(\frac{1+\delta}{2+\delta}\right)^2
\right).
\]
\end{theorem}

\begin{proof}
Set \[
\alpha
=
1-
\min\left(
\frac{\delta}{1+\delta},
\left(\frac{1+\delta}{2+\delta}\right)^2
\right).
\]
By Lemma~\ref{lem:pointwise-2}, we have 
$$K_2^{\mathrm{ad}}(x_1,x_2)\le \alpha$$
whenever $x_1,x_2\ge 0$ and $x_1+x_2\ge 2+\delta.$ Thus, by Lemma~\ref{lemma:reduction} and Theorem~\ref{thm:merger-2}, we have 
\[
\Pp\left(X_1+X_2<2+\delta\right)
\ge1-\alpha=
\min\left(
\frac{\delta}{1+\delta},
\left(\frac{1+\delta}{2+\delta}\right)^2
\right).
\]
\end{proof}

Although Theorem~\ref{thm:n2-arbitrary-delta} is elementary and already known, it is instructive to compare the bounds obtained from \(K_2\) and \(K_2^{\mathrm{ad}}\). For \(0<\delta<1\), the bound obtained from \(K_2\), namely Theorem~\ref{thm:main} with \(n=2\), is not sharp, whereas the bound obtained from \(K_2^{\mathrm{ad}}\) in Theorem~\ref{thm:n2-arbitrary-delta} is sharp for every \(\delta>0\).

To see the sharpness of the first term in Theorem~\ref{thm:n2-arbitrary-delta}, take \(X_1=1\) deterministically and let \(X_2\) equal \(1+\delta\) with probability \(\frac{1}{1+\delta}\) and \(0\) otherwise. Then
\[
\Pp\left(X_1+X_2<2+\delta\right)=\frac{\delta}{1+\delta}.
\]
For the second term, let \(X_1\) and \(X_2\) be independent, with each \(X_i\) equal to \(2+\delta\) with probability \(\frac{1}{2+\delta}\) and \(0\) otherwise. Then
\[
\Pp\left(X_1+X_2<2+\delta\right)=\left(\frac{1+\delta}{2+\delta}\right)^2.
\]

Constructing higher-dimensional analogues of \(K_2^{\mathrm{ad}}\) is substantially more difficult, leaving the arbitrary-\(\delta\) version of Feige's conjecture (Conjecture~\ref{conj:main-2}) open.

\bibliographystyle{alpha}
\bibliography{ref}

\end{document}